\documentclass{elsarticle}
\usepackage{amsmath,amssymb}
\usepackage{amsthm}
\usepackage[a4paper]{geometry}
\usepackage{xcolor}
\usepackage{tikz}
\usetikzlibrary{calc}

\allowdisplaybreaks
\newtheorem{theorem}{Theorem}
\newtheorem{proposition}{Proposition}
\newtheorem{corollary}{Corollary}
\newtheorem{lemma}{Lemma}
\newtheorem{conjecture}{Conjecture}
\newtheorem{example}{Example}

\newtheorem{definition}{Definition}

\theoremstyle{remark}
	
\begin{document}
	\title{On the $2$-binomial complexity of the generalized Thue-Morse words}
    
    
    \author[hzau]{Xiao-Tao L\"{u}}
    \ead{xiaotaoLv@mail.hzau.edu.cn}

    \author[hzau]{Jin Chen}
    \ead{cj@mail.hzau.edu.cn}

    
    \author[hust]{Zhi-Xiong Wen}
    \ead{zhi-xiong.wen@hust.edu.cn}
    
     \author[scut]{Wen Wu\corref{aaa}}
    \ead{wuwen@scut.edu.cn}
    \cortext[aaa]{Corresponding author.}
    
    \address[hzau]{College of Science, Huazhong Agricultural University, Wuhan 430070, China.}
    \address[hust]{School of Mathematics and Statistics, Huazhong University of Science and Technology, Wuhan, 430074, China.}
    \address[scut]{School of Mathematics, South China University of Technology, Guangzhou, 510641, China}
    

	
	\begin{abstract}
		In this paper, we study the $2$-binomial complexity $b_{\mathbf{t}_{m},2}(n)$ of the generalized Thue-Morse words $\mathbf{t}_{m}$ for every integer $m\geq 3$. We obtain the exact value of $b_{\mathbf{t}_{m},2}(n)$ for every integer $n\geq m^{2}$. As a consequence, $b_{\mathbf{t}_{m},2}(n)$ is ultimately periodic with period $m^{2}$. This result partially answers a question of M. Lejeune, J. Leroy and M. Rigo [Computing the $k$-binomial complexity of the Thue-Morse word, J. Comb. Theory Ser. A, {\bf 176} (2020) 105284].
	\end{abstract}

	\maketitle

\section{Introduction}
Abelian equivalence of words has been a subject of great interest for quite a long time.  Given a finite non-empty set $\mathcal{A}$, let $\mathcal{A}^{*}$ and $\mathcal{A}^{\mathbb{N}}$ denote the set of finite words and the set of infinite words  over $\mathcal{A}$ respectively. Two words $u,v\in\mathcal{A}^{*}$ are \emph{abelian equivalent}, denoted by $u\sim_{ab}v$, if $|u|_{a}=|v|_{a}$ for every $a\in\mathcal{A}$ where $|u|_{a}$ denotes the number of occurrences of the letter $a$ in $u$. The notion has been studied in the relation of abelian complexity of infinite words \cite{CW19,LCWW17,MR13,RSZ11}, abelian repetitions and avoidance \cite{CRSZ11,K92,PW20}, and other topics \cite{FMS,PW19,PZ13}; see also \cite{P19} and references therein.


As a generalization of abelian equivalence, Rigo and Salimov \cite{RS15} introduced the notion of $k$-binomial equivalence based on binomial coefficients of words. The \emph{binomial coefficient} $\binom{u}{v}$ of two words $u$ and $v$ is the number of times that $v$ occurs as a (not necessarily contiguous) subsequence of $u$.  Binomial coefficients of finite words have been successfully applied in several fields: $p$-adic topology \cite{BCP89}, non-commutative extension of Mahler's theorem on interpolation series \cite{PS14}, formal language theory \cite{KKS15}, Parikh matrices and  a generalization of Sierpi\'{n}ski's triangle \cite{LRS16}. Many classical questions in combinatorics on words can be considered in the binomial context. Avoiding binomial squares and cubes was considered in \cite{RRS15}. The problem of testing whether two words are $k$-binomially equivalent was discussed in \cite{FJK15}. Let $k\in\mathbb{Z}^{+}\cup \{+\infty\}$. Two words $u$ and $v$ are \emph{$k$-binomially equivalent}, denoted by $u\sim_{k}v$, if $\binom{u}{x}=\binom{v}{x}$ for all words $x$ of length at most $k$. Note that $u\sim_{+\infty} v$ if and only if $u=v$, while $\sim_{1}$ corresponds to the usual notion of abelian equivalence $\sim_{ab}$. Thus one can regard the notion of $k$-binomial equivalence as gradually bridging the gap between abelian equivalence ($k=1$) and equality ($k=+\infty$). An independent generalization of abelian equivalence is $k$-abelian equivalence where one counts factors of length at most $k$ \cite{KSZ13}; for more details, see \cite{R17}. 

Given an infinite word $\mathbf{w}=w(0)w(1)w(2)\cdots\in\mathcal{A}^{\mathbb{N}}$, for every positive integer $n$, let $\mathcal{F}_{\mathbf{w}}(n)$ denote the set of all factors of $\mathbf{w}$ of length $n$. That is, $\mathcal{F}_{\mathbf{w}}(n)=\{w(i)w(i+1)\cdots w(i+n-1) \mid i\geq 0\}$. Set $\rho_{\mathbf{w}}(n)=\sharp (\mathcal{F}_{\mathbf{w}}(n))$. The function $\rho_{\mathbf{w}}$: $\mathbb{Z}^{+}\to\mathbb{Z}^{+}$ is called the factor complexity  function of $\mathbf{w}$.  A fundamental result due to Hedlund and Morse \cite{MH38}
states that an infinite word $\mathbf{w}$ is ultimately periodic if and only if $\rho_{\mathbf{w}}(n)\leq n$ for some $n\geq 1$. Words of factor complexity $\rho_{\mathbf{w}}(n)=n+1$ are called Sturmian words.  Analogously, for each $k\in \mathbb{Z}^{+}\cup \{+\infty\}$, the \emph{$k$-binomial complexity} of $\mathbf{w}$ is define as $b_{\mathbf{w},k}(n)=\sharp(\mathcal{F}_{\mathbf{w}}(n)/\sim_{k})$. The function $b_{\mathbf{w},k}(n)$: $\mathbb{Z}^{+}\to\mathbb{Z}^{+}$ counts the number of $k$-binomial equivalence classes of factors of length $n$ occurring in $\mathbf{w}$. In the case $k=+\infty$, it holds that $b_{\mathbf{w},+\infty}(n)=\rho_{\mathbf{w}}(n)$, while if $k=1$, $b_{\mathbf{w},1}(n)$, denoted by $\rho^{ab}_{\mathbf{w}}(n)$, corresponds to the usual abelian complexity of $\mathbf{w}$.

Abelian complexity is now a widely studied property of infinite words that has been examined for the first time by Coven and Hedlund in \cite{CH73}, where they have revealed that it could serve as an alternative way to characterize periodic words and Sturmian words. Coven and Hedlund showed that an infinite word $\mathbf{w}$ is periodic if and only if its abelian complexity satisfies $\rho^{ab}_{\mathbf{w}}(n)=1$ for all large enough $n$, and they proved that an aperiodic binary infinite word $\mathbf{s}$ is Sturmian if and only if $\rho^{ab}_{\mathbf{s}}(n)=2$ for every integer $n\geq 1$. The notion ``abelian complexity'' itself comes from the paper \cite{RSZ11} which initiated a general study of the abelian complexity of infinite words over finite alphabets. The  abelian complexity functions of some notable words have been determined, for example, the Thue–Morse word \cite{RSZ11}, the paperfolding word \cite{MR13}, the Rudin-Shapiro word \cite{LCWW17} and the generalized Thue-Morse words \cite{CW19}.

However nontrivial infinite words with a closed form of the $k$-binomial complexity $b_{\mathbf{w},k}(n)$ are very rare. There are only a few such examples to date.
\begin{itemize}
	\item Let $\mathbf{t}$ be the Thue-Morse word. Let $k$ be a positive integer. Lejeune, Leroy and Rigo \cite{LLR20} proved that for all $n\leq 2^{k}-1$,  $b_{\mathbf{t},k}(n)=\rho_{\mathbf{t}}(n)$ and  for all $n\geq 2^{k}$, 
	\[b_{\mathbf{t},k}(n)=\begin{cases}
		3\cdot 2^{k}-3,~\text{if}~n\equiv 0~(\mathrm{mod}~2^{k});\\
		3\cdot 2^{k}-4,~\text{otherwise.}\end{cases}\]
	\item For a Sturmian word $\mathbf{s}$,  Rigo and Salimov \cite{RS15} showed that for all $k\geq 2$, $$b_{\mathbf{s},k}(n)=\rho_{\mathbf{s}}(n)=n+1~(n\geq 1).$$ 
	\item Let $\varphi: \mathcal{A}^{*}\to\mathcal{A}^{*}$ be a morphism. If $\varphi(a)\sim_{ab}\varphi(b)$ for all $a,b\in\mathcal{A}$, then $\varphi$ is said to be \emph{Parikh-constant}. Rigo and Salimov \cite{RS15}  proved that there is a constant $C_{\mathbf{x},k}>0$ such that $b_{\mathbf{x},k}(n)\leq C_{\mathbf{x},k}$ for all $n\geq 1$, where $\mathbf{x}$ is a fixed point of a Parikh-constant morphism.
	\item For the Tribonacci word $\mathbf{T}$, Lejeune, Rigo and Rosenfeld
	\cite{LRR20} proved that for all $k\geq 2$, \[b_{\mathbf{T},k}(n)=\rho_{\mathbf{T}}(n)=2n+1~(n\geq 1).\]
\end{itemize}
For every $k\geq 2$, finding explicit value of the $k$-binomial complexity $b_{\mathbf{w},k}(n)$ for a given infinite word $\mathbf{w}$ is a difficult task, particularly in case of words defined over alphabets consisting of more than two letters. In \cite[Section 8] {LLR20}, Lejeune, Leroy and Rigo asked if it is possible to compute  the exact value of $b_{\mathbf{x},k}(n)$ for the fixed point $\mathbf{x}$  of any Parikh-constant morphism such as the generalized Thue-Morse word $\mathbf{t}_{m}$ with $m\geq 3$. Let $\sigma_{m}$ be a morphism over the alphabet $\Sigma_{m}:=\{0,1,\dots, m-1\}$ defined as $0\mapsto 01\cdots(m-1)$, $1\mapsto 12\cdots(m-1)0$, $\dots$, $m-1\mapsto (m-1)0\cdots(m-2)$. The \emph{generalized Thue-Morse word} 
\[\mathbf{t}_{m}:= t_{m}(0)t_{m}(1)t_{m}(2)\dots = 01\cdots\]
is the fixed point of the morphism $\sigma_{m}$ beginning with $0$, i.e., $\mathbf{t}_{m}=\sigma_{m}^{\infty}(0)$. 

The abelian complexity (or $1$-binomial complexity) $b_{\mathbf{t}_{m},1}(n)$ was given in \cite{CW19}.
\begin{theorem}[\cite{CW19}]\label{lem:abel}
	For all integer $n\geq m$, if $n\equiv r \pmod m,$ then
	\[b_{\mathbf{t}_{m},1}(n) = \rho_{\mathbf{t}_{m}}^{ab}(n) =
	\begin{cases} 
		\frac{1}{4}m(m^2-1)+1,  & \text{ if $m$ is odd and $r= 0$};  \\
		\frac{1}{4}m(m-1)^2+m,  & \text{ if $m$ is odd and $r \neq 0$};  \\
		\frac{1}{4}m^3+1,  & \text{ if $m$ is even and $r= 0$}; \\
		\frac{1}{4}m(m-1)^2+\frac{5}{4}m,  & \text{ if $m$ is even and $r\neq 0$ is even}; \\
		\frac{1}{4}m^2(m-2)+m,  & \text{ if $m$ is even and $r\neq 0$ is odd}. 
	\end{cases} \] 
\end{theorem}
In this paper, we fully characterize the $2$-binomial complexity $b_{\mathbf{t}_{m},2}(n)$ with $m\geq 3$. Firstly we give a sufficient and necessary condition of the $2$-binomial equivalence of two factors $u$ and $v$ of the word $\mathbf{t}_{m}$. Let $\mathcal{S}$ and $\mathcal{P}$ denote the set of all the possible suffixes and prefixes defined as
\begin{align*}
	\mathcal{S} & := \{w\in\Sigma_m^{*}\mid  w \text{ is a suffix of }\sigma_m(a) \text{ for some }a\in\Sigma_m \text{ and }|w|<m\},\\
	\mathcal{P} & := \{w\in\Sigma_m^{*}\mid  w \text{ is a prefix of }\sigma_m(a) \text{ for some }a\in\Sigma_m \text{ and }|w|<m\}.
\end{align*}
\begin{theorem}\label{main:result}
	Let $u=\alpha\sigma_{m}(u')\beta$ and $v=\alpha'\sigma_{m}(v')\beta'$ be two factors of the generalized Thue-Morse word $\mathbf{t}_{m}$, where $(\alpha,\beta), (\alpha',\beta')\in\mathcal{S}\times\mathcal{P}$ and  $\min\{|u'|,|v'|\}\geq3$. Then $u\sim_{2}v$ if and only if $\alpha=\alpha'$, $\beta=\beta'$ and $u'\sim_{1} v'$. 
\end{theorem}
We remark that there are infinite words which are fixed points of some Parikh-constant morphisms and do not satisfy the property of Theorem \ref{main:result}. For example, let $\mathbf{x}=\sigma^{\infty}(0)$ where $\sigma$ is the morphism over $\{0,1,2\}$ defined as $0\mapsto 012, 1\mapsto 210, 2\mapsto 120$. Let $u=\sigma(10122)21$ and $v=\sigma(22101)12$. Then $u$ and $v$ are two factors of $\mathbf{x}$ and $u\sim_{2}v$ with $\alpha=\alpha'=\varepsilon$, $\beta=12$ and $\beta'=21$.  It is natural to ask that what kind of words share a property similar to Theorem \ref{main:result}.

Secondly, we obtain the exact value of the $2$-binomial complexity $b_{\mathbf{t}_{m},2}(n)$ for every $n\geq m^{2}$.

\begin{theorem}\label{main:result2}
	
	For every $n\geq m^{2}$ with $m\geq 3$, we have
	\[b_{\mathbf{t}_{m},2}(n)= \begin{cases}
		b_{\mathbf{t}_{m},1}(n/m) + m(m-1)(m(m-1)+1), &\text{ if } n\equiv 0 \pmod m; \\ 
		m^{4}-2m^{3}+2m^{2}, &\text{otherwise,}
	\end{cases}
	\]
	where the abelian complexity $b_{\mathbf{t}_{m},1}(\cdot)$ is given in Theorem \ref{lem:abel}.
\end{theorem}
It follows from Theorem \ref{lem:abel} and Theorem \ref{main:result2} that the $2$-binomial complexity $b_{\mathbf{t}_{m},2}(n)$ is ultimately periodic.
\begin{corollary}
	The $2$-binomial complexity $b_{\mathbf{t}_{m},2}(n)$ of the generalized Thue-Morse word $\mathbf{t}_{m}$ is ultimately periodic with period $m^{2}$.
\end{corollary}

As the abelian complexity (or $1$-binomial complexity) $b_{\mathbf{t}_{m},1}(n)$ and the $2$-binomial complexity $b_{\mathbf{t}_{m},2}(n)$ of the word $\mathbf{t}_{m}$ are ultimately periodic with period $m$ and $m^{2}$ respectively,  our numerical result suggests the following conjecture.

\begin{conjecture}
For every $k\geq 3$, the $k$-binomial complexity $b_{\mathbf{t}_{m},k}(n)$ of the generalized Thue-Morse word is ultimately periodic with period $m^{k}$. 
\end{conjecture}

This paper is organized as follows. In Section $2$, we state some basic definitions and notations. In Section $3$, we prove Theorem \ref{main:result}. In the last section, we compute the exact value of $b_{\mathbf{t}_{m},2}(n)$ for every $n\geq m^{2}$. 
\section{Preliminaries}
\subsection{Finite and infinite words}
An \emph{alphabet} $\mathcal{A}$ is a finite and non-empty set whose elements are called \emph{letters}. Any concatenation of letters from $\mathcal{A}$ is called a word. The concatenation of two words ${u} =u(0)u(1) \cdots u(m)$ and ${v} = v(0)v(1) \cdots v(n)$ is the word ${uv} = u(0)u(1) \cdots u(m)v(0)v(1) \cdots v(n)$. The set of all finite words over $\mathcal{A}$ including the \emph{empty word} $\varepsilon $ is denoted by $\mathcal{A}^*$.  An infinite sequence of letters from $\mathcal{A}$ is called an infinite word and the set of all infinite words over $\mathcal{A}$ is denoted by $\mathcal{A}^{\mathbb{N}}$. For $\omega = \omega(0)\omega(1)\cdots \omega(n-1)\in\mathcal{A}^{*}$, we denote its length by $|\omega|=n$. By convention we set $|\varepsilon|=0$. 

A finite word $w$ is called a \emph{factor} of a finite (or an infinite) word $u$, denoted by $w\prec u$, if there exist a finite word $p$ and a finite (or an infinite) word $s$ such that $u = pws$. We say that the word $w$ is a \emph{prefix} of $u$, denoted by $w\vartriangleleft u$, if $p = \varepsilon$, and a \emph{suffix} of $u$, denoted by $w\vartriangleright u$, if $s = \varepsilon$. Let $\mathbf{w}=w(0)w(1)w(2)\cdots\in\mathcal{A}^{\mathbb{N}}$ be an infinite word. Recall that for every positive integer $n$, the set of all factors of $\mathbf{w}$ of length $n$ is defined as
\[\mathcal{F}_{\mathbf{w}}(n):=\{w(i)w(i+1)\cdots w(i+n-1)\mid i\geq 0\}.\] 
For convenience, we set $\mathcal{F}_{\mathbf{w}}(0)=\{\varepsilon\}$.
Given a finite word $u=u(0)u(1)\cdots u(n-1)$ with $n\geq 2$, we denote the \emph{$1$-length boundary word} consisting of the first and last letter of $u$ by $\partial{u}$, i.e., $\partial{u}=u(0)u(n-1)$. Define $\partial{\mathcal{F}_{\mathbf{w}}}(n) := \{\partial{u}\mid u\in\mathcal{F}_{\mathbf{w}}(n)\}$. For an integer $k\geq 2$ and a finite word $u=u(0)u(1)\cdots u(n)\in\mathcal{A}^{*}$ with $\mathcal{A}\subset \mathbb{N}$, let \[u\equiv u^{\prime}(0)u^{\prime}(1)\cdots u^{\prime}(n-1)\pmod k\] where $u^{\prime}(i)\in\Sigma_{k}$ and $u^{\prime}(i)\equiv u(i)\pmod{k}$ for all $i$. 

\subsection{Binomial coefficients, $k$-binomial equivalence and $k$-binomial complexity}
Now we introduce the binomial coefficients of words, the binomial equivalence of words and the binomial complexity of infinite words. Moreover, we list some properties of binomial coefficients of words. For more details, one can refer to \cite{RS15}.
\begin{definition}$(\emph{Binomial coefficient})$\label{binomial coefficient}
	Let $\mathcal{A}$ be a non-empty finite set and $u=u(0)u(1)\cdots u(n-1)\in\mathcal{A}^n$. Let $s: \mathbb{N}\rightarrow \mathbb{N}$ be an increasing map such that $s(\ell-1)<n$ for all $1\leq \ell\leq n$. Then for $\ell=1,2,\dots,n$, the word $u(s(0))\cdots u(s(\ell-1))$ is a \emph{scattered subword} of length $\ell$ of $u$. The binomial coefficient $\binom{u}{v}$ of two finite words $u$ and $v$ is defined to be the number of times that $v$ occurs as a scattered subword of $u$. In detail,
	\[\binom{u}{v}=\sharp\{(i_{1},i_{2},\dots,i_{|v|})\mid 0\leq i_{1}<i_{2}<\cdots<i_{|v|}\leq n-1, u(i_{1})u(i_{2})\cdots u(i_{|v|})=v\}.\]
\end{definition}
For convenience, let $\binom{u}{v}=0$ if $|v|>|u|$ and $\binom{u}{\varepsilon}=1$.
For example, let $u=u(0)u(1)\cdots u(5)=101000$ and $v=110$. Then $u(0)u(2)u(3)=u(0)u(2)u(4)=u(0)u(2)u(5)=v$.
Hence $\binom{u}{v}=3.$
\begin{definition}$(\emph{$k$-binomial equivalence})$
	Let $k$ be a positive integer and let $\mathcal{A}^{\leq k}$ denote the set of words of length at most $k$ over the alphabet $\mathcal{A}$. We say that $u,v\in\mathcal{A}^{*}$ are \emph{$k$-binomially equivalent} if for every $x\in\mathcal{A}^{\leq k}$,
	\[\binom{u}{x}=\binom{v}{x}.\] 
	We then write $u\sim_{k} v$ if $u$ and $v$ are $k$-binomially equivalent. 
\end{definition}
Indeed, since $\binom{u}{a}=|u|_{a}$ for all $a\in\mathcal{A}$, it is clear that $u\sim_{1} v$ if and only if $u$ and $v$ are abelian equivalent. Note that, for all $k\geq 2$, if $u\sim_{k} v$, then $u\sim_{\ell} v$ for every $1\leq \ell<k$.

There is an equivalent definition of the $k$-binominal equivalence using the extended Parikh vector. Let $u\in\Sigma_{m}^{*}$ and $1\leq \ell\leq k$, there are exact $m^{\ell}$ words of length $\ell$ which can be enumerated lexicographically: $v_{\ell,1},\dots, v_{\ell,m^{\ell}}$. The \emph{extended Parikh vector} of $u$, denoted by $\Psi_{k}(u)$, is 
\[\Psi_{k}(u):=\left(\binom{u}{v_{1,1}},\dots,\binom{u}{v_{1,m}},\binom{u}{v_{2,1}},\dots,\binom{u}{v_{2,m^{2}}},\dots,\binom{u}{v_{k,1}},\dots,\binom{u}{v_{k,m^{k}}}\right).\]
Then $u\sim_{k} v$ if and only if $\Psi_{k}(u)=\Psi_{k}(v)$.
When $k=1$, $\Psi_{1}(u)$ coincides with the Parikh vector of $u$. For convenience, we write $\Psi(u)=\Psi_{1}(u)$.

\begin{example}
	Let $u=010001$ and $v=001010$. Then
	\[\binom{u}{0}=4,\binom{u}{1}=2,\binom{u}{00}=6,
	\binom{u}{01}=5,
	\binom{u}{10}=3,
	\binom{u}{11}=1.\]
	The same equalities hold for the word $v$. Therefore, $\Psi_{2}(u)=\Psi_{2}(v)=(4,2,6,5,3,1)$ and $u\sim_{2} v$.
\end{example}

\begin{definition}$(\emph{$k$-binomial complexity})$ Let $\mathbf{w}$ be an infinite word and let $k$ be a positive integer. The $k$-binomial complexity function $b_{\mathbf{w},k}$ of $\mathbf{w}$ is 
	\[b_{\mathbf{w},k}(n):= \sharp(\mathcal{F}_{\mathbf{w}}(n)/\sim_{k})=\sharp\big\{\Psi_{k}(u)\mid u\in\mathcal{F}_{\mathbf{w}}(n)\}.\]
\end{definition}
In the following, we collect some facts about binomial coefficients and binomial equivalence of words.

\begin{lemma}[\cite{RS15}]\label{prop:binomial}
	\begin{enumerate}[(1)]
		\item Let $u,v$ be two words and let $a,a'$ be two letters. Then 
		\[\binom{au}{a'v}=\binom{u}{a'v}+\delta_{a,a'}\binom{u}{v}~\text{and}~\binom{ua}{va'}=\binom{u}{va'}+\delta_{a,a'}\binom{u}{v},\]
		where $\delta_{a,a'}=1$ if $a=a'$ and $0$ otherwise.
		\item Let $s,w,t$ be three finite words over $\mathcal{A}$ satisfying $|t|\leq |sw|$. Then
		\[\binom{sw}{t}=\sum_{uv=t\text{ with }u,v\in\mathcal{A}^{*}}\binom{s}{u}\binom{w}{v}.\]
		\item Let $u,v,w$ be three words and let $k\geq 1$ be an integer. Then
		\[vu\sim_{k}wu \iff v\sim_{k}w\iff uv\sim_{k}uw.\]
		\item 	Let $k\geq 2$ be an integer and let $u,v,u',v'$ be four words such that $u\sim_{k-1} u'$ but $u\nsim_{k}u'$ and $v\sim_{k}v'$, then $uv\nsim_{k} u'v'$.
	\end{enumerate}
\end{lemma}
\section{$2$-binomial equivalence of factors of the generalized Thue-Morse word}
Fix an integer $m\geq 3$. Recall that $\sigma_{m}$ is the morphism over the alphabet $\Sigma_m=\{0,1,\cdots m-1\}$ defined as $0\mapsto 01\cdots(m-1), 1\mapsto 12\cdots(m-1)0,\cdots, m-1\mapsto (m-1)0\cdots(m-2)$. The infinite word $\mathbf{t}_{m} :=\sigma_{m}^{\infty}(0)$ is called the generalized Thue-Morse word. In this section, we investigate the $2$-binomial equivalence between factors of the generalized Thue-Morse word.

\subsection{Properties of factors of the word $\mathbf{t}_{m}$.}
We collect some basic properties for factors of  $\mathbf{t}_{m}$.

\begin{lemma}\label{alpha:beta:equal}
	Let $\alpha\prec \sigma_{m}(a)$ and $\beta\prec \sigma_{m}(b)$ for some $a,b\in\Sigma_{m}$ with $\max\{|\alpha|,|\beta|\}<m$. Then $\alpha\sim_{1}\beta$ if and only if $\alpha=\beta$. 
\end{lemma}
\begin{proof}
	Suppose $\alpha\sim_1\beta$. By the definition of $\sigma_m$, we can assume that $\alpha=i(i+1)\cdots(i+|\alpha|-1) \pmod{m}$ and $\beta=j(j+1)\cdots (j+|\alpha|-1) \pmod{m}$ for some $i,j\in [0,m-1]$. If $i\ne j$, then there exists $x\in \Sigma_{m}$ such that $|\alpha|_{x}\ne |\beta|_{x}$, which means $\alpha\nsim_{1} \beta$. 	Hence, $i=j$ and $\alpha=\beta$.
\end{proof}\textit{}
For every $c,d\in\Sigma_{m}$ with $c\ne d$, write \[[c,d]:=\begin{cases}
	\{c,c+1,\dots,d\}, & \text{ if }c<d;\\
	\{c,c+1,\dots,m-1\}\cup\{0,1,\dots,d\}, & \text{ if }c>d.
\end{cases}\]
Further, we write $(c,d):=[c,d]\backslash\{c,d\}$. Similarly, we define $(c,d]:=[c,d]\backslash\{c\}$ and $[c,d):=[c,d]\backslash\{d\}$. For example, the intervals $[2,5]$ and $(6,1)$ are illustrated in Figure \ref{1}. 

\begin{figure}  
	\centering  
	\begin{tikzpicture}[scale=0.7]
		\coordinate (O) at (0,0);
		\draw[fill=white] (O) circle (3);
		\draw[fill=white] (O) circle (1);
		 \pgfmathsetmacro\angdiv{360/12}
		 
		\draw [thick,fill=blue!30] ({90-(2-1/2)*\angdiv}:3) -- ({90-(2-1/2)*\angdiv}:1)  arc ({90-(14-1/2)*\angdiv}:{90-(8-1/2)*\angdiv}:1) -- ({90-(8-1/2)*\angdiv}:3)  arc ({90-(8-1/2)*\angdiv}:{90-(14-1/2)*\angdiv}:3) node[text=blue] at ({90-(11-1/2)*\angdiv}:3.7) {\scriptsize{$(6,1)$}};
		
		\draw [thick,fill=red!30] ({90-(7-1/2)*\angdiv}:3) -- ({90-(7-1/2)*\angdiv}:1)  arc ({90-(7-1/2)*\angdiv}:{90-(3-1/2)*\angdiv}:1) --({90-(3-1/2)*\angdiv}:3) arc ({90-(3-1/2)*\angdiv}:{90-(7-1/2)*\angdiv}:3) node[text=red] at ({90-4.5*\angdiv}:3.7) {\scriptsize{$[2,5]$}};

		  \draw
		\foreach \i in {0,...,11}{
				($({90-(\i-1/2)*\angdiv}:3)$) -- ($(({90-(\i-1/2)*\angdiv}:1)$)
			};	
		\foreach \i in {0,...,9}{
		\node at ({45-(\i-1/2)*\angdiv}:2.2) {\scriptsize{$\i$}};
		}
		\node at ({45-(10-1/2)*\angdiv}:2.2) {\scriptsize{$\cdots$}};
		\node at ({45-(11-1/2)*\angdiv}:2.2) {\scriptsize{$m-1$}};
	\end{tikzpicture}
	\caption{Illustration for the intervals $[2,5]$ and $(6,1)$.}  
	\label{1}  
\end{figure}
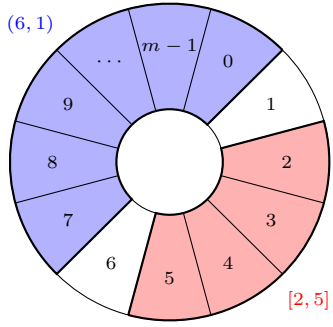

\begin{lemma}\label{e}
	Let  $\gamma\prec\sigma_{m}(a)$ for some $a\in\Sigma_{m}$. Then for every $c,d\in\Sigma_{m}$ with $c\ne d$, we have
	\begin{enumerate}[(1)]
		\item $|\gamma|_{c},|\gamma|_{d}\in\{0,1\}$. Moreover, if $|\gamma|_{c}=|\gamma|_{d}=1$, then either $\binom{\gamma}{cd}=1$ or $\binom{\gamma}{dc}=1$.
		\item If $\binom{\gamma}{cd}=1$, then for every $e\in [c,d]$, we have $|\gamma|_{e}=1$.
		\item If $\binom{\gamma}{cd}=1$ and $|\gamma|<m$, then $(d,c)\neq\emptyset$ and there exists $e\in(d,c)$ such that $|\gamma|_{e}=0$.
	\end{enumerate}
\end{lemma}
\begin{proof}
	The result follows directly from the definition of the morphism $\sigma_{m}$.
\end{proof}
\begin{lemma}\label{abel:a:b}
	Let $\alpha\prec\sigma_{m}(a)$, $\beta\prec\sigma_{m}(b)$, $\alpha'\prec\sigma_{m}(a')$ and $\beta'\prec\sigma_{m}(b')$ for some $a,b,a',b'\in\Sigma_{m}$. If $\alpha\beta\sim_{1}\alpha'\beta'$ and there exist $c,d\in\Sigma_{m}$ with $c\ne d$ such that $\binom{\alpha}{cd}=\binom{\alpha'}{dc}=\binom{\beta'}{dc}=1$, then $|\alpha|=m$.
\end{lemma}
\begin{proof}
	Since $\binom{\alpha'}{dc}=\binom{\beta'}{dc}=1$, by Lemma \ref{e}, for every $e\in[d,c]$, $|\alpha'|_{e}=|\beta'|_{e}=1$. It follows from $\alpha\beta\sim_{1}\alpha'\beta'$ that  $|\alpha|_{e}=1$ for all $e\in [d,c]$. Moreover, $\binom{\alpha}{cd}=1$ implies that $|\alpha|_{e}=1$ with $e\in[c,d]$. Hence, $|\alpha|_{e}=1$ for every $e\in\Sigma_{m}$, which means $|\alpha|=m$.
\end{proof}
\subsection{$2$-binomial equivalence between factors of $\mathbf{t}_m$}

Let $c, d\in\Sigma_{m}$ with $c\neq d$ and $u\in\Sigma_{m}^{*}$. We have
\begin{eqnarray*}
	\binom{\sigma_{m}(u)}{cd}&=&\sum_{0\leq i<j\leq |u|-1}|\sigma_{m}(u(i))|_{c}\cdot|\sigma_{m}(u(j))|_{d}+\sum_{0\leq i\leq |u|-1}\binom{\sigma_{m}(u(i))}{cd}\\
	&=& \binom{|u|}{2}+\sum_{x\in\Sigma_{m},|u|_{x}\geq 1}|u|_{x}\binom{\sigma_{m}(x)}{cd}.
\end{eqnarray*}
Note that for every $x\in\Sigma_{m}$, by the definition of $\sigma_{m}$,
\[\binom{\sigma_{m}(x)}{cd}=\begin{cases}
	1,& x\in(d,c];\\
	0,& \text{otherwise.}
\end{cases}\]
Then 
\begin{equation}\label{u:ab}
	\binom{\sigma_m(u)}{cd} = \binom{|u|}{2}+\sum_{x\in(d,c]}|u|_x.
\end{equation}
Moreover, for all $\alpha,\beta\in\Sigma_{m}^{*}$, 
\begin{align}
	\notag\binom{\alpha\sigma_{m}(u)\beta}{cd}&=\binom{\alpha\beta}{cd}+\binom{\sigma_{m}(u)}{cd}+|\alpha|_{c}|\sigma_{m}(u)|_{d}+|\sigma_{m}(u)|_{c}|\beta|_{d}\\
	&=\binom{\alpha\beta}{cd}+|u|(|\alpha|_{c}+|\beta|_{d})+\sum_{x\in(d,c]}|u|_x+\binom{|u|}{2}\label{alpha:beta:c:d}
\end{align}
where $0\leq \sum_{x\in(d,c]}|u|_x\leq |u|$.

Recall that the sets $\mathcal{S} $ and $\mathcal{P}$ are defined as
\begin{align*}
	\mathcal{S} & := \{w\in\Sigma_m^{*}\mid  w\triangleright\sigma_m(a) \text{ for some }a\in\Sigma_m \text{ and }|w|<m\},\\
	\mathcal{P} & := \{w\in\Sigma_m^{*}\mid  w\triangleleft\sigma_m(a) \text{ for some }a\in\Sigma_m \text{ and }|w|<m\}.
\end{align*}
Now we give the proof of Theorem \ref{main:result}.
\begin{proof}[Proof of Theorem \ref{main:result}]  
	Since $\max\{|\alpha|+|\beta|,|\alpha'|+|\beta'|\}\leq 2m-2<2m$, we have $||u'|-|v'||\leq 1$. Without loss of generality, we may assume $|u'|\geq |v'|$.
	The proof of Theorem \ref{main:result} is separated into the following cases according to the lengths of $\alpha$, $\beta$, $\alpha'$, $\beta'$, $u'$ and $v'$:
	\begin{enumerate}
		\item $|u'|=|v'|$ and $|\alpha|+|\beta|=0$, see Proposition \ref{0:0};
		\item $|u'| =|v'|$ and $|\alpha|+|\beta|\geq 1$, see Propositions \ref{a:b:1} and \ref {a:b:2};
		\item $|u'|=|v'|+1$ and $|\alpha|+|\beta|=0$, see Proposition \ref{a:b:empty};
		\item $|u'|=|v'|+1$, $|\alpha|=0$ and $|\beta|\geq 1$, see Proposition \ref{alpha:0};
		\item $|u'|=|v'|+1$, $|\beta|=0$ and $|\alpha|\geq 1$, see Proposition \ref{beta:0};
		\item $|u'|=|v'|+1$, $|\alpha|\geq 1$ and $|\beta|\geq 1$, see Proposition \ref{a:b:a':b'}. \qedhere
	\end{enumerate}
\end{proof}
\begin{proposition}\label{0:0}
	For every $u',v'\in\Sigma_{m}^{*}$, $\sigma_{m}(u')\sim_{2}\sigma_{m}(v')$ if and only if $u'\sim_{1} v'.$
\end{proposition}
\begin{proof}
	If $u'\sim_{1} v'$, then  $|\sigma_{m}(u')|_{a}=|u'|=|v'|=|\sigma_{m}(v')|_{a}$ for all $a\in\Sigma_{m}$. So, for all $a\in\Sigma_{m}$,
	\[\binom{\sigma_{m}(u')}{aa}=\binom{|\sigma_{m}(u')|_{a}}{2}=\binom{|u'|}{2}=\binom{|v'|}{2}=\binom{|\sigma_{m}(v')|_{a}}{2}=\binom{\sigma_{m}(v')}{aa}.\]
	Moreover, for any $c,d\in\Sigma_{m}$ with $c\ne d$, since $u'\sim_{1}v'$,  by \eqref{u:ab},
	\[
	\binom{\sigma_m(u')}{cd} = \binom{|u'|}{2}+\sum_{x\in(d,c]}|u'|_x = \binom{|v'|}{2}+\sum_{x\in(d,c]}|v'|_x = \binom{\sigma_m(v')}{cd}.
	\]
	Therefore, $\sigma_{m}(u')\sim_{2}\sigma_{m}(v')$.
	
	Conversely, if $\sigma_{m}(u')\sim_{2}\sigma_{m}(v')$, then $|u'|=|v'|$ and $\binom{\sigma(u')}{a(a-1)}=\binom{\sigma(v')}{a(a-1)}$ for all $a\in\Sigma_{m}\backslash\{0\}$. Note that $((a-1),a]=\{a\}$. It follows from \eqref{u:ab} that  
	\[\binom{\sigma_{m}(u')}{a(a-1)}=\binom{|u'|}{2}+\sum_{x\in\{a\}}|u'|_x=\binom{|u'|}{2}+|u'|_{a}\]
	and
	\[\binom{\sigma_{m}(v')}{a(a-1)}=\binom{|v'|}{2}+\sum_{x\in\{a\}}|v'|_x=\binom{|v'|}{2}+|v'|_{a}.\]
	Thus, $|u'|_{a}=|v'|_{a}$ for every $a\in\Sigma_{m}\backslash\{0\}$. Since $|u'|=|v'|$, we have $u'\sim_{1}v'$.
\end{proof}
\begin{proposition}\label{a:b:1}
	Let $(\alpha,\beta)$, $(\alpha',\beta')\in\mathcal{S}\times\mathcal{P}$ with $|\alpha|+|\beta|\geq 1$. Let $u,v\in\mathcal{F}_{\mathbf{t}_{m}}(n)$ with $n\geq 3$.  If $\alpha\sigma_{m}(u)\beta\sim_{2} \alpha'\sigma_{m}(v)\beta'$ and $\alpha\beta\sim_{2}\alpha'\beta'$, then $\alpha=\alpha'$, $\beta=\beta'$ and $u\sim_{1} v$.
\end{proposition}
\begin{proof}
	Since $\alpha\beta\sim_{2}\alpha'\beta'$, $|u|=|v|$ and $\alpha\sigma_{m}(u)\beta\sim_{2} \alpha'\sigma_{m}(v)\beta'$, it follows from \eqref{alpha:beta:c:d} that for every $c\in \Sigma_{m}\backslash\{0\}$, 
	\begin{align*}
		0 & = \binom{\alpha\sigma_{m}(u)\beta}{c(c-1)}-\binom{\alpha\sigma_{m}(v)\beta}{c(c-1)}\\
		&=|u|\bigl(|\alpha|_{c}+|\beta|_{c-1}-|\alpha'|_{c}-|\beta'|_{c-1}\bigr)+|u|_{c}-|v|_{c}.
	\end{align*}
	This implies that $|v|_{c}-|u|_{c}$ is divisible by $|u|$. Note that $|u|=|v|\geq 3$ and the word $\mathbf{t}_{m}$ is cube-free.  This implies $-|u|< |v|_{c}-|u|_{c}<|u|$. Therefore, $|u|_{c}=|v|_{c}$ for all $c\in \Sigma_{m}\backslash\{0\}$. Note also that $|u|=|v|$, we have $u\sim_{1} v$.
	
	For any $c,d\in\Sigma_{m}$, by \eqref{alpha:beta:c:d}, we have
	\begin{equation}
		\binom{\alpha\sigma_{m}(u)\beta}{cd}  = \binom{\alpha\beta}{cd}+|u|(|\alpha|_c+|\beta|_d)+\binom{\sigma_{m}(u)}{cd}\label{prop:3-1}
	\end{equation}
	and 
	\begin{equation}
		\binom{\alpha'\sigma_{m}(v)\beta'}{cd}  = \binom{\alpha'\beta'}{cd}+|v|(|\alpha'|_{c}+|\beta'|_{d})+\binom{\sigma_{m}(v)}{cd}.\label{prop:3-2}
	\end{equation}
	Since $u\sim_{1}v$, by Proposition \ref{0:0}, we have $\sigma_{m}(u)\sim_{2}\sigma_{m}(v)$ and $\binom{\sigma_{m}(u)}{cd}=\binom{\sigma_{m}(v)}{cd}$. Further, since $\alpha\sigma_{m}(u)\beta\sim_{2} \alpha'\sigma_{m}(v)\beta'$ and  $\alpha\beta\sim_{2}\alpha'\beta'$, we have \[\binom{\alpha'\sigma_{m}(u)\beta'}{cd}=\binom{\alpha\sigma_{m}(v)\beta}{cd}~\text{and}~  \binom{\alpha\beta}{cd}=\binom{\alpha'\beta'}{cd}.\] Note that $|u|=|v|$. It follows from \eqref{prop:3-1} and \eqref{prop:3-2} that for any $c,d\in\Sigma_{m}$,
	\begin{equation}\label{prop:3-3}
		|\alpha|_c+|\beta|_d=|\alpha'|_{c}+|\beta'|_{d}.
	\end{equation}
	Now we prove $\alpha=\alpha'$ and $\beta=\beta'$. There are two sub-cases. 
	\begin{enumerate}[(i)]
		\item $|\alpha|\geq 1$ and $|\beta|\geq 1$. For any $c,d\in\Sigma_{m}$ satisfying $|\alpha|_{c}=|\beta|_{d}=1$, by \eqref{prop:3-3}, we have  
		\[|\alpha'|_{c}+|\beta'|_{d}=2.\]
		Since $(\alpha',\beta')\in\mathcal{S}\times\mathcal{P}$, we see that  $|\alpha'|_{c}=|\beta'|_{d}=1$. In the same way, one can show that $|\alpha|_{c}=|\beta|_{d}=1$ for any $c,d\in\Sigma_{m}$ satisfying $|\alpha'|_{c}=|\beta'|_{d}=1$. This implies that   $\alpha\sim_{1}\alpha'$ and $\beta\sim_{1}\beta'$. By Lemma \ref{alpha:beta:equal}, $\alpha=\alpha'$ and $\beta=\beta'$.	
		
		\item $|\alpha|=0$ or $|\beta|=0$.  It suffices to show the case $|\alpha|=0$. 
		
		If $\beta=d$ for some $d\in\Sigma_{m}$, then $\alpha'\beta'\sim_{2}\alpha\beta$ implies that $\alpha'\beta'=d$. For any $c\in\Sigma_{m}$ with $c\ne d$, by \eqref{prop:3-3}, $1=|\alpha'|_{c}+|\beta'|_{d}$. So, $\alpha'=\varepsilon\text{ and }\beta'=d.$
		
		If $|\beta|>1$, then there exists a $d\in\Sigma_m$ with $|\beta|_d=1$. For any $c\in \Sigma_{m}$ with $c\neq d$, by \eqref{prop:3-3}, we have
		\begin{equation}\label{c:0}
			|\alpha'|_{c}+|\beta'|_{d}=|\beta|_{d}=1.
		\end{equation}
		Since $\alpha\beta\sim_{2}\alpha'\beta'$, which means $\alpha\beta\sim_{1}\alpha'\beta'$, we have 
		\begin{equation}\label{c:1}
			|\alpha'|_{d}+|\beta'|_{d}=|\beta|_{d}=1.
		\end{equation}
		Combining \eqref{c:0} and \eqref{c:1}, we obtain that for all $x\in\Sigma_m$, $|\alpha'|_x=1-|\beta'|_d.$ If $|\beta'|_d=0$, then $|\alpha'|=m$ which contradicts to the fact $|\alpha'|<m$. If $|\beta'|_d=1$, then $\alpha'=\varepsilon$. Therefore, $\beta=\alpha\beta\sim_1\alpha'\beta'=\beta'$. By Lemma \ref{alpha:beta:equal}, $\beta=\beta'$. \qedhere
	\end{enumerate}
\end{proof}
\begin{proposition} \label{a:b:2}
	Let $(\alpha,\beta)$, $(\alpha', \beta')\in\mathcal{S}\times\mathcal{P}$ with $|\alpha|+|\beta|\geq 1$. If $\alpha\beta\nsim_{2}\alpha'\beta'$, then for any $u,u'\in \Sigma_{m}^{*}$ with $|u|=|u'|$, $\alpha\sigma_{m}(u)\beta\nsim_{2}\alpha'\sigma_{m}(u')\beta'$.
\end{proposition}
\begin{proof}
	Since $|u|=|u'|$, we have $\sigma_m(u)\sim_1\sigma_m(u')$. If $\alpha\beta\nsim_{1}\alpha'\beta'$, then $\alpha\sigma_{m}(u)\beta\nsim_{1}\alpha'\sigma_{m}(u')\beta'$. In the rest, we assume that $\alpha\beta\sim_{1}\alpha'\beta'$. 
	
	Since $\alpha\beta\nsim_{2}\alpha'\beta'$ and $\alpha\beta\sim_{1}\alpha'\beta'$, there exist $c, d \in\Sigma_{m}$ with $c\neq d$ such that $\binom{\alpha\beta}{cd}\neq \binom{\alpha'\beta'}{cd}$. Note that 
	\begin{equation}\label{alpha:beta:c:d:2}
		\binom{\alpha\beta}{cd}=\binom{\alpha}{cd}+\binom{\beta}{cd}+|\alpha|_{c}|\beta|_{d}.
	\end{equation}
	This implies that $\binom{\alpha\beta}{cd},\binom{\alpha'\beta'}{cd}\in\{0,1,2,3\}$. Without loss of generality, we can assume that $\binom{\alpha\beta}{cd}>\binom{\alpha'\beta'}{cd}$.  Then, $\binom{\alpha\beta}{cd}\in\{1,2,3\}$. If $\binom{\alpha\beta}{cd}=3$, then  it follows from \eqref{alpha:beta:c:d:2} that $\binom{\alpha}{cd}=\binom{\beta}{cd}=1$ and $|\alpha\beta|_c=|\alpha\beta_d|=2$. Since $\alpha\beta\sim_1\alpha'\beta'$, we have $|\alpha'\beta'|_c=|\alpha'\beta'|_d=2$. Note also that $\binom{\alpha'\beta'}{cd}<3$, by Lemma \ref{e}, at least one of $\binom{\alpha'}{dc}$ and $\binom{\beta'}{dc}$ equals $1$. By Lemma \ref{abel:a:b}, $|\alpha'|=m$ or $|\beta'|=m$ which contradicts to the fact that $(\alpha',\beta')\in\mathcal{S}\times\mathcal{P}$. So $\binom{\alpha\beta}{cd}\neq 3$.
	
	In the following, we deal with the cases that $\binom{\alpha\beta}{cd}=1$ or $2$.
	By \eqref{alpha:beta:c:d},
	\begin{align*}
		\binom{\alpha\sigma_{m}(u)\beta}{cd} & \geq \binom{\alpha\beta}{cd}+|u|(|\alpha|_{c}+|\beta|_{d})+\binom{|u|}{2},\\
		\binom{\alpha'\sigma_{m}(u')\beta'}{cd} & \leq \binom{\alpha'\beta'}{cd}+|u'|(|\alpha'|_{c}+|\beta'|_{d})+\binom{|u'|}{2}+|u'|.
	\end{align*}
	Since $|u|=|u'|$ and $\binom{\alpha\beta}{cd}>\binom{\alpha'\beta'}{cd}$, we have 
	\begin{equation}\label{prop:4-1}
		\binom{\alpha\sigma_{m}(u)\beta}{cd} -\binom{\alpha'\sigma_{m}(u')\beta'}{cd}>|u|\bigl(|\alpha|_c+|\beta|_d-|\alpha'|_c-|\beta'|_d-1\bigr).
	\end{equation}
	
	When $\binom{\alpha\beta}{cd}=2$, by \eqref{alpha:beta:c:d:2}, we have $|\alpha|_c+|\beta|_d=2$. If $|\alpha'|_c|\beta'|_d=0$, then $|\alpha'|_c+|\beta'|_d\leq 1$. It follows from \eqref{prop:4-1} that $\binom{\alpha\sigma_{m}(u)\beta}{cd}>\binom{\alpha'\sigma_{m}(u')\beta'}{cd}$ and $\alpha\sigma_{m}(u)\beta\nsim_2\alpha'\sigma_{m}(u')\beta'$. If $|\alpha'|_c|\beta'|_d=1$, then $2=\binom{\alpha\beta}{cd}>\binom{\alpha'\beta'}{cd}$ implies $\binom{\alpha'}{cd}=\binom{\beta'}{cd}=0$. Since $\binom{\alpha\beta}{cd}=2$, following \eqref{alpha:beta:c:d:2}, we see that either $\binom{\alpha}{cd}=1$ or $\binom{\beta}{cd}=1$. 
	\begin{enumerate}
		\item If $\binom{\alpha}{cd}=1$, then $|\alpha\beta|_d=2$. Since $\alpha\beta\sim_1\alpha'\beta'$, we have $|\alpha'\beta'|_d=2$ which implies $|\alpha'|_d=1$. Recall that $|\alpha'|_c=1$ and $\binom{\alpha'}{cd}=0$. We have $\binom{\alpha'}{dc}=1$. By Lemma \ref{e}, there exists an $e\in(c,d)\neq \emptyset$ such that $|\alpha'|_e=0$ and $|\alpha|_e=1$. Moreover,  $\alpha\beta\sim_1\alpha'\beta'$ yields that $|\beta'|_e=1$ and $|\beta|_e=0$. Then $\binom{\alpha\beta}{ed}=2>\binom{\alpha'\beta'}{ed}$ and $|\alpha|_e+|\beta|_d-|\alpha'|_e-|\beta'|_d=1+1-0-1=1$. Applying \eqref{prop:4-1}, we have $\binom{\alpha\sigma_{m}(u)\beta}{ed}>\binom{\alpha'\sigma_{m}(u')\beta'}{ed}$.
		\item If $\binom{\beta}{cd}=1$, then $|\alpha\beta|_c=2$. It follows from $\alpha\beta\sim_1\alpha'\beta'$ and $\binom{\beta'}{cd}=0$ that $\binom{\beta'}{dc}=1$. By Lemma \ref{e}, there exists an $e\in(c,d)\neq \emptyset$ such that $|\beta'|_e=0$ and $|\beta|_e=1$. Moreover,  $\alpha\beta\sim_1\alpha'\beta'$ yields that $|\alpha'|_e=1$ and $|\alpha|_e=0$. Then $\binom{\alpha\beta}{ce}=2>\binom{\alpha'\beta'}{ce}$ and $|\alpha|_c+|\beta|_e-|\alpha'|_c-|\beta'|_e=1+1-1-0=1$. Applying \eqref{prop:4-1}, we have $\binom{\alpha\sigma_{m}(u)\beta}{ce}>\binom{\alpha'\sigma_{m}(u')\beta'}{ce}$.
	\end{enumerate}
	
	When $\binom{\alpha\beta}{cd}=1$, we have $\binom{\alpha'\beta'}{cd}=0$. Then $|\alpha|_c+|\beta|_d\geq 1\geq|\alpha'|_c+|\beta'|_d$. If $\bigl(|\alpha|_c+|\beta|_d\bigr)-\bigl(|\alpha'|_c+|\beta'|_d\bigr)\geq 1$, then by \eqref{prop:4-1}, we have $\binom{\alpha\sigma_{m}(u)\beta}{cd}>\binom{\alpha'\sigma_{m}(u')\beta'}{cd}$. Now suppose that $|\alpha|_c+|\beta|_d=|\alpha'|_c+|\beta'|_d=1$. There are four sub-cases.
	\begin{enumerate}
		\item $|\alpha|_c=1,|\beta|_d=0$. Then, it follows from $\binom{\alpha\beta}{cd}=1$ that $|\alpha|_d=1$ and $\binom{\alpha}{cd}=1$. 
		\begin{itemize}
			\item If $|\alpha'|_c=1$ and $|\beta'|_d=0$, then $\alpha\beta\sim_1\alpha'\beta'$ implies that $|\alpha'|_d=1$ and $\binom{\alpha'}{dc}=1$. By Lemma \ref{e}, there exists an $e\in(c,d)\neq \emptyset$ such that $|\alpha'|_e=0$ and $|\alpha|_e=1$. Then $\binom{\alpha\beta}{ed}=1>\binom{\alpha'\beta'}{ed}=0$ and $|\alpha|_e+|\beta|_d-|\alpha'|_e-|\beta'|_d=1+0-0-0=1$. Applying \eqref{prop:4-1}, we have $\binom{\alpha\sigma_{m}(u)\beta}{ed}>\binom{\alpha'\sigma_{m}(u')\beta'}{ed}$.
			\item If $|\alpha'|_c=0$ and $|\beta'|_d=1$. Since $\alpha\beta\sim_1\alpha'\beta'$, we have $|\beta'|_c=1$ and $\binom{\beta'}{dc}=1$. By Lemma \ref{e}, there exists an $e\in(c,d)\neq \emptyset$ such that $|\beta'|_e=0$ and $|\alpha|_e=1$. Moreover, $\alpha\beta\sim_1\alpha'\beta'$ yields that $|\alpha'|_e=1$ and $|\beta|_e=0$. Then $\binom{\alpha\beta}{ce}=1>\binom{\alpha'\beta'}{ce}=0$ and $|\alpha|_c+|\beta|_e-|\alpha'|_c-|\beta'|_e=1+0-0-0=1$. Applying \eqref{prop:4-1}, we have $\binom{\alpha\sigma_{m}(u)\beta}{ce}>\binom{\alpha'\sigma_{m}(u')\beta'}{ce}$.
		\end{itemize}
		\item $|\alpha|_c=0,|\beta|_d=1$. It follows from $\binom{\alpha\beta}{cd}=1$ that $|\beta|_c=1$ and $\binom{\beta}{cd}=1$. 
		\begin{itemize}
			\item If $|\alpha'|_c=1$ and $|\beta'|_d=0$, then $\alpha\beta\sim_1\alpha'\beta'$ yields $|\alpha|_d=0$, $|\alpha'|_d=1$ and $\binom{\alpha'}{dc}=1$. By Lemma \ref{e}, there exists an $e\in(d,c)$ such that $|\alpha'|_e=1$ and $|\beta|_e=0$. Since $\alpha\beta\sim_1\alpha'\beta'$, we have $|\beta'|_e=0$ and $|\alpha|_e=1$. Then $\binom{\alpha'\beta'}{de}=1>\binom{\alpha\beta}{de}=0$ and $|\alpha'|_d+|\beta'|_e-|\alpha|_d-|\beta|_e=1+0-0-0=1$. Applying \eqref{prop:4-1}, we have $\binom{\alpha'\sigma_{m}(u')\beta'}{de}>\binom{\alpha\sigma_{m}(u)\beta}{de}$.
			\item If $|\alpha'|_c=0$ and $|\beta'|_d=1$, then $\binom{\beta'}{dc}=1$. Using a similar argument, there exists an $e\in(c,d)$ such that $|\beta|_e=1$ and $|\beta'|_e=0$. Consequently, $\binom{\alpha\sigma_{m}(u)\beta}{ce}>\binom{\alpha'\sigma_{m}(u')\beta'}{ce}$.
		\end{itemize}
	\end{enumerate}
\end{proof}

\begin{proposition}\label{a:b:empty}
	Let $(\alpha,\beta)\in\mathcal{S}\times\mathcal{P}$ and $\alpha, \beta\neq\varepsilon$. Then for any $u\in\mathcal{F}_{\mathbf{t}_{m}}(n)$ and $v\in\mathcal{F}_{\mathbf{t}_{m}}(n+1)$ with $n\geq 3$, $\alpha\sigma_{m}(u)\beta\nsim_{2}\sigma_{m}(v)$.
\end{proposition}
\begin{proof}
	Suppose that $v=v'c$ where $c\in \Sigma_{m}$. Since $|u|=|v'|$,  we have $\sigma_{m}(u)\sim_{1}\sigma_{m}(v')$. If $\alpha\beta\nsim_{1}\sigma_{m}(c)$, then $\alpha\sigma_{m}(u)\beta\nsim_{1}\sigma_{m}(v)$.  Now we assume that $\alpha\beta\sim_{1}\sigma_{m}(c)$. There are two cases.
	\begin{enumerate}[(i)]
		\item $\alpha\beta=\sigma_{m}(c)$. Let $d$ be the last letter of $\alpha$ and let $e$ be the first letter of $\beta$. By the definition of $\sigma_{m}$, $e\equiv d+1\pmod{m}$ and $|\alpha|_{e}=|\beta|_{d}=0$. By \eqref{alpha:beta:c:d}, 
		\[\binom{\alpha\sigma_{m}(u)\beta}{ed}=\binom{\sigma_{m}(u)}{ed}=\binom{|u|}{2}+\sum_{x\in(d,e]}|u|_{x}=\binom{|u|}{2}+|u|_{e}\]
		and 
		\[\binom{\sigma_{m}(v'c)}{ed}=\binom{\sigma_{m}(v')}{ed}+|v'|=\binom{|v'|}{2}+|v'|_{e}+|v'|.\]
		Since the word $\mathbf{t}_{m}$ is cube-free and $|u|\geq 3$, $|u|_{e}<|u|=|v'|$. Hence $\binom{\alpha\sigma_{m}(u)\beta}{ed}<\binom{\sigma_{m}(v'c)}{ed}$ and $\alpha\sigma_{m}(u)\beta\nsim_{2}\sigma_{m}(v'c)$.
		\item $\alpha\beta\ne\sigma_{m}(c)$. There exist $d,e\in\Sigma_{m}$ such that $|\alpha|_{d}=1=|\beta|_{e}$ and  $\binom{\sigma_{m}(c)}{de}=0$. Note that $|\alpha|_{e}=|\beta|_{d}=0$. By \eqref{alpha:beta:c:d}, 
		\[\binom{\alpha\sigma_{m}(u)\beta}{de}\geq 1+2|u|+\binom{|u|}{2}>2|v'|+\binom{|v'|}{2}\geq \binom{\sigma_{m}(v'c)}{de}.\]
		Therefore $\alpha\sigma_{m}(u)\beta\nsim_{2}\sigma_{m}(v'c)$.\qedhere
	\end{enumerate}
\end{proof}

\begin{proposition}\label{alpha:0}
	Let $\beta$, $\beta'\in\mathcal{P}$ and $\alpha'\in\mathcal{S}$ where $\beta$, $\alpha'$ and $\beta'$ are nonempty words. Then for any $u\in\mathcal{F}_{\mathbf{t}_{m}}(n+1)$ and $v\in\mathcal{F}_{\mathbf{t}_{m}}(n)$ with $n\geq 3$, we have 
	$\sigma_{m}(u)\beta\nsim_{2}\alpha'\sigma_{m}(v)\beta'$.
\end{proposition}
\begin{proof}
	Let $u=au'$ where $a\in\Sigma_{m}$. If $\alpha'\beta'\nsim_{1}\sigma_{m}(a)\beta$, then $\sigma_{m}(au')\beta\nsim_{1}\alpha'\sigma_{m}(v)\beta'$. In the following, we assume that $\sigma_{m}(a)\beta\sim_{1}\alpha'\beta'$. Then, for every $x\in\Sigma_{m}$ with $|\beta|_{x}=1$, we have $|\alpha'|_{x}=|\beta'|_{x}=1$. Noting that $\sigma_{m}(a)\beta\sim_{1}\alpha'\beta'$, we see that $|\beta'|>|\beta|$.

	Let $b$ and $c$ be the first and the last letter of $\beta$ respectively. (If $|\beta|=1$, then $b=c$.) Since $\sigma_{m}(a)\beta\sim_{1}\alpha'\beta'$, we have $|\alpha'|_b=|\beta'|_b=|\alpha'|_c=|\beta'|_c=1$. 
	\begin{itemize}
		\item If $b$ is the first letter of $\beta'$, then $c$ is not the last letter of $\beta'$. Moreover, letting $d\equiv c+1\pmod{m}$, we have $|\beta|_d=0$ and $|\beta'|_d=1$. It follows from $\sigma_{m}(a)\beta\sim_{1}\alpha'\beta'$ that $|\alpha'|_d=0$. Then $\binom{\sigma_m(a)\beta}{cd}\leq 1<2=\binom{\alpha'\beta'}{cd}$ and $|\alpha'|_c+|\beta'|_d-|\sigma_m(a)|_c-|\beta|_d=1+1-1-0=1$. Using \eqref{prop:4-1}, we have $\binom{\sigma_{m}(au')\beta}{cd}<\binom{\alpha'\sigma_{m}(v)\beta'}{cd}$ and  $\sigma_{m}(au')\beta\nsim_{2}\alpha'\sigma_{m}(v)\beta'$.
		\item If $b$ is not the first letter of $\beta'$, then $e\equiv b-1\pmod{m}$ occurs in $\beta'$ and $|\alpha'|_e=0$. Now $\binom{\sigma_m(a)\beta}{eb}\geq 1=\binom{\alpha'\beta'}{eb}$ and $|\sigma_m(a)|_e+|\beta|_b-|\alpha'|_e-|\beta'|_b=1+1-0-1=1$. By the fact that the word $\mathbf{t}_{m}$ is cube-free and $|u'|\geq 3$, we have $\sum_{x\in(b,e]}|u'|_{x}=|u'|-|u'|_b>0$. Using \eqref{alpha:beta:c:d}, we have $\binom{\sigma_{m}(au')\beta}{eb}>\binom{\alpha'\sigma_{m}(v)\beta'}{eb}$ and  $\sigma_{m}(au')\beta\nsim_{2}\alpha'\sigma_{m}(v)\beta'$. \qedhere
	\end{itemize}
\end{proof}

\begin{proposition}\label{beta:0}
	Let $\alpha,\alpha'\in\mathcal{S}$ and $\beta'\in\mathcal{P}$ where $\alpha$, $\alpha'$ and $\beta'$ are nonempty. Then for any $u\in\mathcal{F}_{\mathbf{t}_{m}}(n+1)$ and $v\in\mathcal{F}_{\mathbf{t}_{m}}(n)$ with $n\geq 3$, we have 
	$\alpha\sigma_{m}(u)\nsim_{2}\alpha'\sigma_{m}(v)\beta'$.
\end{proposition}
\begin{proof}
	Let $u=u'b$ for some $b\in\Sigma_{m}$. If $\alpha\sigma_{m}(b)\nsim_{1}\alpha'\beta'$, then  $\alpha\sigma_{m}(u)\nsim_{1}\alpha'\sigma_{m}(v)\beta'$. Now we assume that $\alpha\sigma_{m}(b)\sim_{1}\alpha'\beta'$. Then $|\alpha'|>|\alpha|$ and for every $x\in\Sigma_{m}$ with $|\alpha|_{x}=1$, $|\alpha'|_{x}=|\beta'|_{x}=1$.
	Let $c$ and $d$ be the first and last letter of $\alpha$ respectively. (If $|\alpha|=1$, then $c=d$.)
	\begin{itemize}
		\item If $c$ is the first letter of $\alpha'$, then set $e\equiv d+1\pmod{m}$. Consequently, $\binom{\alpha'}{de}=1$ and $|\alpha|_{e}=0$. By $\alpha\sigma_{m}(b)\sim_{1}\alpha'\beta'$, we have $|\beta'|_{e}=0$. Note that  $\binom{\alpha\sigma_m(b)}{de}\geq 1=\binom{\alpha'\beta'}{de}$, $|\alpha|_d+|\sigma_m(b)|_e-|\alpha'|_d-|\beta'|_e=1+1-1-0=1$. It follows from the word $\mathbf{t}_{m}$ is cube-free and $|u|\geq 3$ that $\sum_{x\in(e,d]}|u'|_x=|u'|-|u'|_e>0$. By \eqref{alpha:beta:c:d} and $|u'|=|v|$, we have  \[\binom{\alpha\sigma_{m}(u'b)}{de}>1+2|u'|+\binom{|u'|}{2}\geq \binom{\alpha'\sigma_{m}(v)\beta'}{de}.\] Hence $\alpha\sigma_{m}(u'b)\nsim_{2}\alpha'\sigma_{m}(v)\beta'$.
		\item If $c$ is not the first letter of $\alpha'$, then set $e\equiv c-1\pmod{m}$. Then $\binom{\alpha'}{ec}=1$ and $|\alpha|_{e}=0$. Now $\binom{\alpha'\beta'}{ec}=2>1\geq\binom{\alpha\sigma_m(b)}{ec}$ and $|\alpha'|_e+|\beta'|_c-|\alpha|_e-|\sigma_m(b)|_c=1+1-0-1=1$. By \eqref{prop:4-1}, $\binom{\alpha'\sigma_{m}(v)\beta'}{ec}>\binom{\alpha\sigma_{m}(u'b)}{ec}$. So, $\alpha\sigma_{m}(u'b)\nsim_{2}\alpha'\sigma_{m}(v)\beta'$. \qedhere
	\end{itemize}
\end{proof}

\begin{proposition}\label{a:b:a':b'}
	Let $(\alpha,\beta),(\alpha',\beta')\in\mathcal{S}\times\mathcal{P}$ and $\min\{|\alpha|,|\beta|,|\alpha'|,|\beta'|\}\geq 1$. Then for any $u\in\mathcal{F}_{\mathbf{t}_{m}}(n+1)$, $v\in\mathcal{F}_{\mathbf{t}_{m}}(n)$ with $n\geq 3$, we have
	$\alpha\sigma_{m}(u)\beta\nsim_{2}\alpha'\sigma_{m}(v)\beta'$.
\end{proposition}
\begin{proof}
	Let $u=u'c$ for some $c\in\Sigma_{m}$. If $\alpha\sigma_{m}(c)\beta\nsim_{1}\alpha'\beta'$, then $\alpha\sigma_{m}(u'c)\beta\nsim_{1}\alpha'\sigma_{m}(v)\beta'$. From now on, we assume that $\alpha\sigma_{m}(c)\beta\sim_{1}\alpha'\beta'$. Then for all $x\in\Sigma_m$, $|\alpha|_x+|\beta|_x=|\alpha\sigma_m(c)\beta|_x-1=|\alpha'\beta'|_x-1\leq 1$. Consequently, there exist $d,e\in\Sigma_m$ with $d\neq e$ such  that $|\alpha|_{d}=|\beta|_{e}=1$. Hence $|\beta|_d=|\alpha|_e=0$ and $|\alpha'|_{d}=|\beta'|_{d}=|\alpha'|_{e}=|\beta'|_{e}=1$.	By \eqref{alpha:beta:c:d},
	\begin{align*}
		\binom{\alpha\sigma_{m}(u)\beta}{de}
		&\geq \binom{\alpha\beta}{de}+|u|(|\alpha|_{d}+|\beta|_{e})+\binom{|u|}{2}
		=3+3|u'|+\binom{|u'|}{2}
	\end{align*}
	and 
	\begin{align*}
		\binom{\alpha'\sigma_{m}(v)\beta'}{de}
		&\leq \binom{\alpha'\beta'}{de}+3|v|+\binom{|v|}{2}
		=\binom{\alpha'\beta'}{de}+3|u'|+\binom{|u'|}{2}.
	\end{align*}
	Note that $\binom{\alpha'\beta'}{de}\leq 3$. If $\binom{\alpha'\beta'}{de}<3$, then $\binom{\alpha\sigma_{m}(u)\beta}{de}>\binom{\alpha'\sigma_{m}(v)\beta'}{de}$ and $\alpha\sigma_{m}(u)\beta\nsim_{2}\alpha'\sigma_{m}(v)\beta'$. If $\binom{\alpha'\beta'}{de}=3$, then $\binom{\alpha'}{de}=\binom{\beta'}{de}= 1$. For every $x\in[d,e]$, $|\alpha'|_{x}=|\beta'|_{x}=1$ and $|\alpha\beta|_x=1$ since $\alpha\sigma_{m}(c)\beta\sim_{1}\alpha'\beta'$. Then, it follows from $(\alpha,\beta)\in\mathcal{S}\times\mathcal{P}$ and $|\alpha|_{d}=|\beta|_{e}=1$ that there exists $y\in[d,e)$ such that $y\triangleright\alpha$ and $z\triangleleft\beta$ where $z\equiv (y+1)\pmod{m}$. As the word $\mathbf{t}_{m}$ is cube-free and $|u|>3$, we have $\sum_{x\in(z,y]}|u|_x=|u|-|u|_z>0$. By \eqref{alpha:beta:c:d},
	\[\binom{\alpha\sigma_{m}(u)\beta}{yz}
	> \binom{\alpha\beta}{yz}+|u|(|\alpha|_{y}+|\beta|_{z})+\binom{|u|}{2}
	=3+3|u'|+\binom{|u'|}{2}\geq \binom{\alpha'\sigma_{m}(v)\beta'}{yz}.\]
	Hence, $\alpha\sigma_{m}(u)\beta\nsim_{2}\alpha'\sigma_{m}(v)\beta'$.
\end{proof}

\section{$2$-binomial complexity of the generalized Thue-Morse word}
The aim of this section is to compute the $2$-binomial complexity of the generalized Thue-Morse word $\mathbf{t}_{m}$.  For every $u\in\Sigma_{m}^{*}$, the Parikh vector of $u$ is denoted by
\[\Psi(u):=(|u|_{0},|u|_{1},\cdots,|u|_{m-1}).\]
Write $\mathbf{1}:=(1,1,\dots,1)=\Psi(\sigma_m(0))$. For $x\in\Sigma_m$ and $\ell\in\mathbb{N}$, define $x|_{\ell}\in\Sigma_m^{\ell}$ as follow \[x|_{\ell}\equiv x(x+1)\dots(x+\ell-1)\pmod{m}.\]  If $\ell=0$, then $x|_{\ell}:=\varepsilon$. 

To compute the $2$-binomial complexity $b_{\mathbf{t}_{m},2}(n)$, we need some auxiliary lemmas.

	\begin{lemma}[Lemma 2 in \cite{CW19}]\label{bound:t}
		For the generalized Thue-Morse sequence $\mathbf{t}_{m}$ and every integer $n\geq 2$, we have
		$\partial{\mathcal{F}_{\mathbf{t}_{m}}}(n)=\Sigma_{m}^{2}.$
	\end{lemma}

\begin{lemma}\label{0:a}
	Let $n=km$ with $k\geq 1$. Then for every $a\in\Sigma_{m}$, we have
	\[\sharp\{\Psi(au)\mid au\in\mathcal{F}_{\mathbf{t}_{m}}(n)\}=\sharp\{\Psi(ua)\mid ua\in\mathcal{F}_{\mathbf{t}_{m}}(n)\}=1+\frac{m(m-1)}{2}.\]
\end{lemma}
\begin{proof}
	For $i\in[0,m-1]$ and $b\in\Sigma_m$, write 
	\[\mathcal{E}_{i,b}:=\{\alpha\sigma_m(v)\beta\in\mathcal{F}_{\mathbf{t}_m}(n)\mid \alpha=a|_i, \beta=b|_{m-i}, v\in\mathcal{F}_{\mathbf{t}_m}(k-1), a\triangleleft \alpha\sigma_{m}(v)\}.\]
	By Lemma \ref{bound:t}, for every $i\in[0,m-1]$ and $b\in\Sigma_{m}$, there exists some $v\in\Sigma_{m}^{*}$ such that $\mathcal{E}_{i,b}\ne\emptyset$.
	Let $\mathcal{E}_i=\cup_{b\in\Sigma_m}\mathcal{E}_{i,b}$. Then $\{w\in\mathcal{F}_{\mathbf{t}_m}(n)\mid a\triangleleft w\}=\cup_{j=0}^{m-1}\mathcal{E}_j$. Note that 
	\begin{align*}
		E_{i,b}:=\{\Psi(w) \mid w\in\mathcal{E}_{i,b}\} & = \{\Psi(a|_i)+\Psi(b|_{m-i})+\Psi(\sigma_m(v))\mid v\in\mathcal{F}_{\mathbf{t}_m}(k-1)\}\\
		& = \{\Psi(a|_i)+\Psi(b|_{m-i})+(k-1)\mathbf{1}\}.
	\end{align*}
	Write $E_i:=\{\Psi(w)\mid w\in\mathcal{E}_i\}$. Then $\sharp E_0=1$ and for $i\in[1,m-1]$, $\sharp E_i=m$. Observe that for any $i\in[1,m-1]$ and $b\in\Sigma_m$, 
	\[\begin{cases}
		E_{i,b}=E_0, & \text{ if }b-a-i\equiv 0\pmod{m};\\
		E_{i,b}=E_{j,b}, & \text{ if }b-a-i\equiv j\in[1,i-1];\\
		E_{i,b}\cap\left(\cup_{\ell=0}^{i-1}E_{\ell}\right)=\emptyset, & \text{ otherwise}.
	\end{cases}\]
	So, for $i\in[1,m-1]$, $\sharp\left(E_i\setminus\left(\cup_{j=0}^{i-1}E_j\right)\right)=m-i$. Hence \[\sharp\{\Psi(au)\mid au\in\mathcal{F}_{\mathbf{t}_{m}}(n)\}=\sharp\left(\cup_{j=0}^{m-1}E_i\right)=1+(m-1)+(m-2)+\dots+1=1+\frac{m(m-1)}{2}.\]

	
	Applying a similar argument, one has $\sharp\{\Psi(ua)\mid ua\in\mathcal{F}_{\mathbf{t}_{m}}(n)\}=1+\frac{m(m-1)}{2}$.
\end{proof}

\begin{lemma}\label{a:r:a}
	Let $n=km+r$ with $k\geq 1$ and $1\leq r\leq m-1$. Then for every $a\in\Sigma_{m}$, we have
	\[\sharp\{\Psi(au)\mid au\in\mathcal{F}_{\mathbf{t}_{m}}(n)\}=\sharp\{\Psi(ua)\mid ua\in\mathcal{F}_{\mathbf{t}_{m}}(n)\}=1+\frac{m(m-1)}{2}.\]
\end{lemma}
\begin{proof}
	For $i\in[0,r]$ and $b\in\Sigma_m$, write \[\mathcal{E}_{i,b}:=\{\alpha\sigma_m(v)\beta\in\mathcal{F}_{\mathbf{t}_m}(n)\mid \alpha=a|_{i}, \beta=b|_{r-i}, v\in\mathcal{F}_{\mathbf{t}_m}(k), a\triangleleft \alpha\sigma_m(v)\}\]and $\mathcal{E}_i:=\cup_{b\in\Sigma_m}\mathcal{E}_{i,b}$. By Lemma \ref{bound:t}, $\mathcal{E}_{i,b}\ne \emptyset$ for every $b\in\Sigma_{m}$ and $i\in[0,m-1]$. Let $E_i=\cup_{b\in\Sigma_m}E_{i,b}$ where  
	\[E_{i,b}:=\{\Psi(w)\mid w\in\mathcal{E}_{i,b}\}=\{\Psi(a|_{i})+\Psi(b|_{r-i})+k\mathbf{1}\}.\]
	Moreover, for $i\in[0,r-1]$, $\sharp E_i=m$ and $\sharp E_r=1$. Observe that $E_r\subset E_0$ and for $i\in[0,r-1]$, 
	\[\begin{cases}
		E_{i,b}=E_{0,a}, & \text{if }b\equiv a+i\pmod{m};\\
		E_{i,b}=E_{j,b}, & \text{if }b+r-i-a\equiv j \in[0,i-1];\\
		E_{i,b}\cap\left(\cup_{\ell=0}^{i-1}E_{\ell}\right)=\emptyset, & \text{otherwise}.
	\end{cases}\]
	Thus, for $i\in[1,r-1]$, $\sharp\left(E_i\setminus\cup_{\ell=0}^{i-1}E_{\ell}\right)=m-i-1$. So, $\sharp\left(\cup_{i=0}^{r}E_i\right)=1+\sum_{i=1}^{r}(m-i)$.
	
	For $i\in[1,m-r-1]$ and $b\in\Sigma_m$, write \[\mathcal{D}_{i,b}:=\{\alpha\sigma_m(v)\beta\in\mathcal{F}_{\mathbf{t}_m}(n)\mid \alpha=a|_{m-i}, \beta=b|_{r+i}, v\in\mathcal{F}_{\mathbf{t}_m}(k-1)\}\] and $\mathcal{D}_i:=\cup_{b\in\Sigma_m}\mathcal{D}_{i,b}$. By Lemma \ref{bound:t}, we have $\mathcal{D}_{i,b}\ne\emptyset$ for all $b\in\Sigma_{m}$ and $i\in[1,m-r-1]$. Let $D_i=\cup_{b\in\Sigma_m}D_{i,b}$ where  
	\[D_{i,b}:=\{\Psi(w)\mid w\in\mathcal{D}_{i,b}\}=\{\Psi(a|_{m-i})+\Psi(b|_{r+i})+(k-1)\mathbf{1}\}.\]
	Note that $\sharp D_i=m$. Further, for $i\in[1,m-r-1]$ and $j\in\Sigma_m$, letting $b\equiv a+m-i-j\pmod{m}$, we have 
	\[\begin{cases}
		D_{i,b} = E_{r-j,b}, & \text{if } j\in[0,r];\\
		D_{i,b} = D_{j-r,b}, & \text{if }j\in[r+1,r+i-1];\\
		D_{i,b}\cap\left[\bigl(\cup_{\ell=0}^{r}E_{\ell}\bigr)\cup\bigl(\cup_{s=1}^{i-1}D_s\bigl)\right]=\emptyset, & \text{if }j\in[r+i,m-1].
	\end{cases}\]
	Since $\{\Psi(au)\mid au\in\mathcal{F}_{\mathbf{t}_{m}}(n)\}=\bigl(\cup_{i=0}^{r}E_i\bigr)\cup\bigl(\cup_{i=1}^{m-r-1}D_i\bigr)$, we obtain that  	
	\[\sharp\left[\bigl(\cup_{i=0}^{r}E_i\bigr)\cup\bigl(\cup_{i=1}^{m-r-1}D_i\bigr)\right]=1+\sum_{i=1}^{r}(m-i)+\sum_{i=r+1}^{m-1}(m-i)=1+\frac{m(m-1)}{2}.\]
	
	Applying a similar argument, one has $\sharp\{\Psi(ua)\mid ua\in\mathcal{F}_{\mathbf{t}_{m}}(n)\}=1+\frac{m(m-1)}{2}$.
\end{proof}

\begin{lemma}\label{a:b:n}
	For every $a,b\in\Sigma_{m}$ and $n\in\mathbb{N}$ with $n\geq m+1$, we have 
	\[\sharp\{\Psi(aub)\mid aub\in\mathcal{F}_{\mathbf{t}_{m}}(n)\}=\begin{cases}
		1,~\text{if}~n\equiv b-a+1 \pmod m;\\
		m,~\text{otherwise}.
	\end{cases}\]
\end{lemma}
\begin{proof}
	Fix $a$, $b\in\Sigma_m$. Suppose $n=km$ for some $k\geq 2$. Adopt the notations in the proof of Lemma \ref{0:a}. Recall that $\{w\in\mathcal{F}_{\mathbf{t}_m}(n)\mid a\triangleleft w\}=\cup_{i=0}^{m-1}\mathcal{E}_i$. Note that for $i\in[0,m-1]$, $\{w\in\mathcal{E}_i\mid b\triangleright w\}=\mathcal{E}_{i,b+i+1}$ and \[E_{i,b+i+1}=\{\Psi(w)\mid w\in\mathcal{E}_{i,b+i+1}\}=\{\Psi(a|_i)+\Psi((b+i+1)|_{m-i})+(k-1)\mathbf{1}\}.\] If $b+i+1\equiv a+i\pmod{m}$, then for all $i\in[0,m-1]$, $E_{i,b+i+1}=\{k\mathbf{1}\}$. Consequently, $\sharp\{\Psi(aub)\mid aub\in\mathcal{F}_{\mathbf{t}_{m}}(n)\}=\sharp E_{0,a}=1$. If $b+i+1\not\equiv a+i\pmod{m}$, then for any $i,j\in[0,m-1]$ with $i\neq j$, $E_{i,b+i+1}\neq E_{j,b+j+1}$. In this case, $\sharp\{\Psi(aub)\mid aub\in\mathcal{F}_{\mathbf{t}_{m}}(n)\}=\sharp\left(\cup_{i=0}^{m-1}E_{i,b+i+1}\right)=m$.
	
	Now suppose $n=km+r$ where $k\geq 1$ and $1\leq r<m$. Adopt the notations in the proof of Lemma \ref{a:r:a}. Recall that $\{w\in\mathcal{F}_{\mathbf{t}_m}(n)\mid a\triangleleft w\}=\bigl(\cup_{i=0}^{r}\mathcal{E}_i\bigr)\cup\bigl(\cup_{i=1}^{m-r-1}\mathcal{D}_i\bigr)$. Again, for $i\in[0,r]$, $\{w\in\mathcal{E}_i\mid b\triangleright w\}=\mathcal{E}_{i,b-r+i+1}$ and for $j\in[1,m-r-1]$, $\{w\in\mathcal{D}_i\mid b\triangleright w\}=\mathcal{D}_{i,b-r-i+1}$. Moreover, \[\{\Psi(aub)\mid aub\in\mathcal{F}_{\mathbf{t}_{m}}(n)\}=\bigl(\cup_{i=0}^{r}E_{i,b-r+i+1}\bigr)\cup\bigl(\cup_{i=1}^{m-r-1}D_{i,b-r-i+1}\bigr).\]
	If $b-a+1\equiv r\pmod{m}$, then for all $i\in[0,r]$ and $j\in[1,m-r-1]$, \[E_{i,b-r+i+1}=D_{j,b-r-j+1}=\{\Psi(a|_{r})+k\mathbf{1}\}.\] Therefore, $\sharp\{\Psi(aub)\mid aub\in\mathcal{F}_{\mathbf{t}_{m}}(n)\}=1$.	If $b-a+1\not\equiv r\pmod{m}$, then 
	\[E_{0,b-r+1}, ~E_{1,b-r+2}, ~\dots, ~E_{r,b+1},~ D_{1,b-r},~ D_{2,b-r-1}, ~\dots, ~D_{m-r-1,b-m+2}~\text{are~pairwise~disjoint}.\] 
	 Thus $\sharp\{\Psi(aub)\mid aub\in\mathcal{F}_{\mathbf{t}_{m}}(n)\}=m.$
\end{proof}
Recall that for $u\in\Sigma_m^*$, the Parikh vector and the extended Parikh vector of $u$ are denoted by $\Psi(u)$ and $\Psi_{k}(u)$ respectively. The following two theorems give the accurate value of $b_{\mathbf{t}_{m},2}(n)$ for every $n\geq m^{2}$. 
\begin{theorem}\label{2:n:0}
	For every $k\geq m$, we have
	\[b_{\mathbf{t}_{m},2}(km) = b_{\mathbf{t}_{m},1}(k) + m(m-1)[m(m-1)+1]\]
	where $b_{\mathbf{t}_{m},1}(\cdot)$ denotes  the abelian complexity function of the infinite word $\mathbf{t}_{m}$.
\end{theorem}
\begin{proof}
	Observe that $\mathcal{F}_{\mathbf{t}_{m}}(km)=F_1\cup F_2$ with $F_1\cap F_2=\emptyset$ where 
	\begin{align*}
		F_{1} & := \{\sigma_{m}(v)\mid v\in\mathcal{F}_{\mathbf{t}_{m}}(k)\},\\
		F_{2} & := \{\alpha\sigma_{m}(v)\beta\in\mathcal{F}_{\mathbf{t}_{m}}(km) \mid (\alpha,\beta)\in\mathcal{S}\times\mathcal{P}, |\alpha|+|\beta|=m,v\in\mathcal{F}_{\mathbf{t}_m}(k-1)\}.
	\end{align*}
	Applying Theorem \ref{main:result}, we have  
	\begin{align}
		b_{\mathbf{t}_{m},2}(km) & = \sharp\{\Psi_{2}(u)\mid u\in\mathcal{F}_{\mathbf{t}_{m}}(km)\}\nonumber\\
		& = \sharp\{\Psi_{2}(u)\mid u\in F_{1}\}+\sharp\{\Psi_{2}(u)\mid u\in F_{2}\}.\label{thm:2-1}
	\end{align}
	Using Theorem \ref{main:result} again, for $u=\sigma_m(v)\in F_1$ and $u'=\sigma_m(v')\in F_1$,  $u\sim_2 u'$ if and only if $v\sim_1 v'$, i.e., $\Psi_{2}(u)=\Psi_{2}(u')\iff\Psi(v)=\Psi(v')$.
	Therefore, 
	\begin{equation}\label{thm:2-2}
		\sharp\{\Psi_{2}(u)\mid u\in F_1\}=\sharp\{\Psi(v)\mid v\in\mathcal{F}_{\mathbf{t}_m}(k)\}=b_{\mathbf{t}_{m},1}(k).
	\end{equation}
	
	For $i\in[1,m-1]$, write \[F_{2,i,a,b}:=\{\alpha\sigma_m(v)\beta\in F_2\mid \alpha=a|_i,b\triangleright\beta,v\in\mathcal{F}_{\mathbf{t}_m}(k-1)\}.\] Then 
	$F_2=\cup_{i=1}^{m-1}\cup_{a,b\in\Sigma_m}F_{2,i,a,b}$. 
	According to Theorem \ref{main:result}, 
	\begin{align*}
		& \sharp\{\Psi_{2}(w)\mid w\in F_{2,i,a,b}\} \\
		= & \sharp\{\Psi(cvd)\mid cvd\in\mathcal{F}_{\mathbf{t}_{m}}(k+1),~c\equiv a+i,~ d\equiv b+i+1~(\bmod~m)\}\\
		= & \begin{cases}
			1, & \text{if }b\equiv a+k-1\pmod{m};\\
			m, & \text{otherwise.}
		\end{cases} \tag{\text{by Lemma \ref{a:b:n}}}
	\end{align*}
	Using Theorem \ref{main:result} again, if $(i,a,b)\neq (i',a',b')$, then for any $u\in F_{2,i,a,b}$ and $u'\in F_{2,i',a',b'}$, we have $u\nsim_2 u'$. So, 
	\begin{align}
		\sharp\{\Psi_{2}(u)\mid u\in F_{2}\} & = \sum_{i=1}^{m-1}\sum_{a,b\in\Sigma_m}\sharp\{\Psi_{2}(w)\mid w\in F_{2,i,a,b}\}\nonumber\\
		& = (m-1)m[1+(m-1)m].\label{thm:2-3}
	\end{align}
	Then the result follows from \eqref{thm:2-1}, \eqref{thm:2-2} and \eqref{thm:2-3}.
\end{proof}

\begin{theorem}\label{2:n:r}
	For every $n\geq m^{2}$ with $n\not\equiv 0 \pmod m$, we have
	\[b_{\mathbf{t}_{m},2}(n)=m^{4}-2m^{3}+2m^{2}.\]
\end{theorem}
\begin{proof}
	Let $n=km+r$ for some $k\geq m$ and $1\leq r\leq m-1$. Then the set $\mathcal{F}_{\mathbf{t}_{m}}(n)$ can be separated into four disjoint parts:
	\begin{align*}
		S_{1} & := \{\alpha\sigma_{m}(v)\in \mathcal{F}_{\mathbf{t}_{m}}(n) \mid \alpha=a|_r,a\in\Sigma_{m},v\in\mathcal{F}_{\mathbf{t}_{m}}(k)\},\\
		S_{2} & := \{\sigma_{m}(v)\beta\in \mathcal{F}_{\mathbf{t}_{m}}(n)\mid \beta=b|_r,b\in\Sigma_{m},v\in\mathcal{F}_{\mathbf{t}_{m}}(k)\},\\
		S_{3} & := \{\alpha\sigma_{m}(v)\beta\in \mathcal{F}_{\mathbf{t}_{m}}(n) \mid \alpha=a|_i, \beta=b|_{r-i}, i\in[1,r-1], a,b\in\Sigma_{m},v\in\mathcal{F}_{\mathbf{t}_{m}}(k)\},\\
		S_{4} & := \{\alpha\sigma_{m}(v)\beta\in \mathcal{F}_{\mathbf{t}_{m}}(n) \mid \alpha=a|_i, \beta=b|_{m+r-i}, i\in[r+1,m-1], a,b\in\Sigma_{m}, v\in\mathcal{F}_{\mathbf{t}_{m}}(k-1)\}.
	\end{align*}
	According to Theorem \ref{main:result}, 	
	\begin{align}
		\sharp\{\Psi_{2}(w)\mid w\in S_1\} & = \sum_{a\in\Sigma_m}\sharp\{\Psi_{2}(a|_r\sigma_{m}(v)) \mid v\in\mathcal{F}_{\mathbf{t}_{m}}(k),a|_r\sigma_{m}(v)\in \mathcal{F}_{\mathbf{t}_{m}}(n)\}\nonumber\\
		& = \sum_{a\in\Sigma_m}\sharp\{\Psi(a'v)\mid a'v\in\mathcal{F}_{\mathbf{t}_{m}}(k+1), a'\equiv a+r\pmod{m}\}\nonumber\\
		& = m\left(1+\frac{m(m-1)}{2}\right),\label{thm:3-1}
	\end{align}
	where in the last step, we use Lemma \ref{a:r:a}. Similarly, we have 
	\begin{equation}\label{thm:3-2}
		\sharp\{\Psi_{2}(w)\mid w\in S_2\} = m\left(1+\frac{m(m-1)}{2}\right).
	\end{equation}
	
	For $i\in[1,r-1]$, write \[S_{3,i,a,b}:=\left\{\alpha\sigma_{m}(v)\beta\in \mathcal{F}_{\mathbf{t}_{m}}(n) \mid \alpha=a|_i, \beta=b|_{r-i},v\in\mathcal{F}_{\mathbf{t}_{m}}(k)\right\}.\] Then $S_3=\cup_{i=1}^{r-1}\cup_{a,b\in\Sigma_m}S_{3,i,a,b}$.
	By Lemma \ref{a:b:n}, \[\sharp\{\Psi(a'vb)\mid a'vb\in\mathcal{F}_{\mathbf{t}_{m}}(k+2),~a'\equiv a+i~(\bmod~{m})\}=\begin{cases}
			1, & \text{if }b\equiv a+i+k+1\pmod{m};\\
			m, & \text{otherwise}.
		\end{cases}\]
	It follows from Theorem \ref{main:result},
	\begin{align}
		\sharp\{\Psi_{2}(w)\mid w\in S_3\}  & = \sum_{i=1}^{r-1}\sum_{a,b\in\Sigma_m}\sharp{\{\Psi(a'vb)\mid a'vb\in\mathcal{F}_{\mathbf{t}_{m}}(k+2),~a'\equiv a+i\pmod{m}\}}\nonumber\\
		& = (r-1)m[1+m(m-1)]. \label{thm:3-3}
	\end{align}
	For $i\in[r+1,m-1]$ and $ a,b\in\Sigma_{m}$, write 
	\[S_{4,i,a,b} := \{\alpha\sigma_{m}(v)\beta\in \mathcal{F}_{\mathbf{t}_{m}}(n) \mid \alpha=a|_i, \beta=b|_{m+r-i}, v\in\mathcal{F}_{\mathbf{t}_{m}}(k-1)\}.\]
	By Theorem \ref{main:result} and Lemma \ref{a:b:n}, 
	\begin{align}
			\sharp\{\Psi_{2}(w)\mid w\in S_4\} & = \sum_{i=r+1}^{m-1}\sum_{a,b\in\Sigma_m}\sharp\{\Psi(a'vb)\mid a'vb\in\mathcal{F}_{\mathbf{t}_{m}}(k+1),~a'\equiv a+i\pmod{m}\}\nonumber\\
			& = (m-r-1)m[1+m(m-1)].\label{thm:3-4}
		\end{align}
	Combining \eqref{thm:3-1}, \eqref{thm:3-2}, \eqref{thm:3-3}, \eqref{thm:3-4} and Theorem \ref{main:result}, we have 
	\begin{align*}
		b_{\mathbf{t}_m,2}(n) & = \sharp\{\Psi_{2}(w)\mid w\in\mathcal{F}_{\mathbf{t}_m}(n)\}\\
		& = \sum_{j=1}^4\sharp\{\Psi_{2}(w)\mid w\in S_j\}
		= m^4-2m^3+2m^2. \qedhere
	\end{align*}
\end{proof}
\begin{proof}[Proof of Theorem \ref{main:result2}] 
The result follows from Theorem \ref{2:n:0} and Theorem \ref{2:n:r} directly. 
\end{proof}

\section*{Acknowledgement}
This work was supported by NSFC (Nos. 11801203, 11701202, 11871295), Guangzhou Science and Technology program (202102020294), Guangdong Basic and Applied Basic Research Foundation (2021A1515010056) and the Fundamental Research Funds for the Central Universities from SCUT (2020ZYGXZR041).

\end{document}